\documentclass[11pt,a4paper]{amsart}

\usepackage{amsthm}
\usepackage{latexsym,bezier,caption,amsbsy,psfrag,color,enumerate,amsfonts,amsmath,amscd,amssymb,comment}
\usepackage{epsfig}
\usepackage{rotating}
\usepackage{graphics}

\DeclareMathOperator{\St}{Stab}

\newtheorem{theo}{Theorem}[section]
\newtheorem{theorem}[theo]{Theorem}
\newtheorem{prop}[theo]{Proposition}
\newtheorem{defn}[theo]{Definition}
\newtheorem{lemma}[theo]{Lemma}
\newtheorem{cor}[theo]{Corollary}
\newtheorem{rem}[theo]{Remark}
\newtheorem{example}[theo]{Example}

\newcommand{\la}{\langle}
\newcommand{\ra}{\rangle}
\newcommand{\wt}{\widetilde}

\def\vlongrightarrow{\relbar\joinrel\longrightarrow}
\def\vvlongrightarrow{\relbar\joinrel\vlongrightarrow}
\def\vvvlongrightarrow{\relbar\joinrel\vvlongrightarrow}

\def\vlongmapright#1{\smash{\mathop{\vvlongrightarrow}\limits^{#1}}}

\newcommand{\vlongfr}[2]{\smash{\stackrel{\text{\tiny{$#1|#2$}\,}}{\vvlongrightarrow}}}
\newcommand{\vvlongfr}[2]{\smash{\stackrel{\text{\tiny{$#1|#2$}\,}}{\vvvlongrightarrow}}}

\newcommand{\mapright}[1]{\smash{\stackrel{\text{\tiny{$#1$}}}{\vlongrightarrow}}}

\begin{document}
\keywords{Tree automaton group, Dual automaton, Tree automaton semigroup, Reducible automaton, Poly-context-free group.}

\title[On a class of poly-context-free groups generated by automata]{On a class of poly-context-free groups generated by automata}

\author{Matteo Cavaleri}
\address{Matteo Cavaleri, Universit\`{a} degli Studi Niccol\`{o} Cusano - Via Don Carlo Gnocchi, 3 00166 Roma, Italia}
\email{matteo.cavaleri@unicusano.it}

\author{Daniele D'Angeli}
\address{Daniele D'Angeli, Universit\`{a} degli Studi Niccol\`{o} Cusano - Via Don Carlo Gnocchi, 3 00166 Roma, Italia}
\email{daniele.dangeli@unicusano.it \qquad \textrm{(Corresponding Author)}}

\author{Alfredo Donno}
\address{Alfredo Donno, Universit\`{a} degli Studi Niccol\`{o} Cusano - Via Don Carlo Gnocchi, 3 00166 Roma, Italia}
\email{alfredo.donno@unicusano.it}

\author{Emanuele Rodaro}
\address{Emanuele Rodaro, Politecnico di Milano - Piazza Leonardo da Vinci, 32 20133 Milano, Italia}
\email{emanuele.rodaro@polimi.it}

\begin{abstract}
This paper deals with graph automaton groups associated with trees and some generalizations. We start by showing some algebraic properties of tree automaton groups. Then we characterize the associated semigroup, proving that it is isomorphic to the partially commutative monoid associated with the complement of the line graph of the defining tree. After that, we generalize these groups by introducing the quite broad class of reducible automaton groups, which lies in the class of contracting automaton groups without singular points. We give a general structure theorem that shows that all reducible automaton groups are direct limit of poly-context-free groups which are virtually subgroups of the direct product of free groups; notice that this result partially supports a conjecture by T. Brough. Moreover, we prove that tree automaton groups with at least two generators are not finitely presented and they are amenable groups, which are direct limit of non-amenable groups.
\end{abstract}

\maketitle

\begin{center}
{\footnotesize{\bf Mathematics Subject Classification (2020)}: 20E08, 20F05, 20F10, 20M35, 68Q70.}
\end{center}

\section{Introduction}

Automaton groups constitute a remarkable and exciting class of groups generated by finite transducers also called Mealy machines. This class has gained popularity thanks to the Grigorchuk group, introduced in 1980 in \cite{gri80} as the first example of a group of intermediate (i.e., faster than polynomial and slower than exponential) growth, answering to an important question posed by Milnor. Over the last decades, this class has been shown to have deep connections with the theory of profinite groups, with combinatorics via the notion of Schreier graph, and with complex dynamics via the notion of iterated monodromy group (for more details see, for instance, \cite{fractal_gr_sets, bhn:aut_til, dynamicssubgroup, nekrashevyc}). In the last years, a special interest has been pointed out for decision problems for automaton groups and semigroups (see \cite{DaRo14, freeness, israel, gillibert, gill, con} and references therein). Many automaton groups exhibit very interesting and exotic properties. In fact this class contains, for instance, examples of Burnside groups, of amenable but not elementary amenable groups, of groups with intermediate growth. Although there are many well studied examples of automaton groups and a quite extensive literature, very little is known from their very general structural point of view. It is easy to show that such groups must be residually finite and that they must have solvable word problem, but besides such properties there exists no general classification. Even just focusing on some specific subclasses of automaton groups (for instance branch, fractal, just-infinite, e.g., \cite{am, pipa}) it seems to be a hard task to obtain structural results.

This paper has the aim of studying some properties of a class of automaton groups, called tree automaton groups, belonging to the family of graph automaton groups introduced by the authors in \cite{articolo0} (and further studied in \cite{articolo1}) and of framing such class into a more general class of automaton groups having some very peculiar topological properties.

In Section \ref{sectiondue} we recall some preliminary definitions and properties about automaton groups and dual automata, whereas in Section \ref{sectiontre} we recall the basic definition of graph automaton group introduced in \cite{articolo0}.\\
\indent In Section~\ref{sec: some properties} we obtain some algebraic properties of such groups: we prove that they are not solvable (Corollary \ref{coro: non-solvable}), they have trivial center (Corollary \ref{tree}), and we provide some necessary conditions on the structure of their relations (Theorem \ref{relazioni}), with an application to the context of tree automaton groups. \\
\indent In Section~\ref{sec: semigroups} we study the semigroups defined by the generators of a tree automaton group and we prove that they are all partially commutative monoids. The graph defining the partial commutations is obtained by taking the complement of the line graph of the tree $T$ associated with the tree automaton group (Theorem \ref{theo: submonoid is a trace monoid}). It is a remarkable fact that this result does not depend on the particular orientation of the defining tree. On the other hand, when the graph defining the automaton group is not a tree, the semigroup structure does depend on the orientation of the edges, as observed in Remark \ref{interesse}. \\
\indent Finally, in Section~\ref{sec: a general structure theorem}, inspired by a property about relations that holds for tree automaton groups, we define a quite broad class of automaton groups, generated by automata that we call reducible. This class sits in the intersection between automaton groups without singular points and contracting automaton groups. Besides containing tree automaton groups, it contains for instance the famous Basilica group, introduced by R. Grigorchuk and A. \.{Z}uk in \cite{grizuk1} as a group generated by a three state automaton. At the end of the section, we also exhibit a procedure to generate other examples of such automata. We prove a general structure theorem which shows that a group generated by a reducible automaton is isomorphic to the direct limit of virtually subgroups of the direct product of free groups; such subgroups either are abelian or they contain a free non-abelian group (Theorem \ref{thmastrazeneca} and Proposition \ref{prop: dichotomy}). As a consequence, we are able to show that a group generated by a reducible automaton and having exponential growth, is a direct limit of non-amenable groups (Theorem \ref{thmnonome}). Moreover, under the further assumption that such a group is amenable, our results imply that it is not finitely presented (Theorem \ref{presentability}). We also show that the family of groups appearing in the direct limit are deterministic poly-context-free groups, a class of groups introduced by Brough \cite{Tara} for which it has been conjectured to be virtually subgroups of the direct product of free groups. Although we do not prove this conjecture, we support it by showing that the groups appearing in this limit satisfy this conjecture thanks to the fact that the words representing the identity constitute a context-free language arising by taking the inverse image of a Dyck language by a finite transducer.

\section{Preliminaries on automaton groups}\label{sectiondue}
Let $X$ be a finite set. For each integer $n\geq 1$, let $X^n$ (resp. $X^{\geq n}$) denote the set of words of length $n$ (resp. of length greater than or equal to $n$) over the alphabet $X$ and put $X^0  = \{\emptyset\}$, where $\emptyset$ is the empty word. Moreover put $X^\ast= \bigcup_{n=0}^\infty X^n$ and denote by $X^\infty = \{x_1x_2x_3\ldots : x_i\in X\}$ the set of infinite words over $X$.

A finite \textit{automaton} is a quadruple $\mathcal{A} =(A,X,\lambda,\mu)$, where:
\begin{enumerate}
\item $A$ is a finite set, called the set of \emph{states};
\item $X$ is a finite set, called the \emph{alphabet};
\item $\lambda: A\times X \rightarrow A$ is the \emph{restriction} map;
\item $\mu: A\times X \rightarrow X$ is the \emph{output} map.
\end{enumerate}
Observe that many authors refer to such an automaton as a deterministic automaton. Given a state $s\in A$ and an element $x\in X$, we put:
$$
s\circ x =\mu(s,x)\in X  \qquad \qquad s\cdot x=\lambda(s,x) \in A.
$$
The automaton $\mathcal{A}$ is \emph{invertible} if, for all $s\in
A$, the transformation $s\circ:X\rightarrow X$ is a permutation of $X$. An automaton $\mathcal{A}$ can be visually
represented by its \textit{Moore diagram}: this is a directed labeled graph whose vertices are identified with the states of
$\mathcal{A}$. For every state $s\in A$ and every letter $x\in X$, the diagram has an arrow from $s$ to $s\cdot x$ labeled by $x|s\circ x$. A sink $id$ in $\mathcal{A}$ is a state with the property that $id\circ x=x$ and $id\cdot x=id$ for every $x\in X$. One can visualize an invertible automaton by using such a directed graph: for any $s\in A$ and $x\in X$ there is exactly one transition of the form
$$
s\vlongmapright{x|s\circ x}s\cdot x.
$$
For an automaton $\mathcal{A} =(A,X,\lambda,\mu)$ we may define its square automaton $\mathcal{A}^2=(A^2, X, \lambda', \mu')$ having the transition $s_1s_2\mapright{x|y}t_1t_2$ whenever $s_1\mapright{x|z}t_1$ and $s_2\mapright{z|y}t_2$ are transitions in $\mathcal{A}$. In a similar fashion we may define the $n$-th power $\mathcal{A}^n$ of $\mathcal{A}$ for each $n$.

The action of $\mathcal{A}$ can be naturally extended to the infinite set $X^\ast$ and to the set $X^\infty$ of infinite words over $X$. Moreover one can compose the action of the states in $A$ extending the maps $\circ$ and $\cdot$ to the set $A^{\ast}$. More precisely, given $w=s_1\cdots s_n\in A^*$ and
$u\in X^{\ast}$ we have:
\begin{eqnarray}\label{papaakiev}
w\circ u=(s_2\ldots s_n)\circ (s_1\circ u) \qquad w\cdot u=\left((s_2\ldots s_n)\cdot (s_1\circ u) \right)(s_1\cdot u).
\end{eqnarray}
Given an invertible automaton $\mathcal{A}$, the \textit{automaton group} generated by $\mathcal{A}$ is by definition the group generated by the bijective transformations $s\circ $ of $X^\ast$, for $s\in A$, and it is denoted by $G(\mathcal{A})$. We refer the interested reader to the monograph \cite{nekrashevyc} for a more comprehensive discussion.                                                                                   When we consider the generating set $A$ without inverses, we get a semigroup $S(\mathcal{A})$,  called the \emph{automaton semigroup} generated by $\mathcal{A}$.
Notice that the action of $G(\mathcal{A})$ (and so $S(\mathcal{A})$) on $X^\ast$ preserves the sets $X^n$, for each $n$.
Moreover, it is not difficult to check that the maps in Eqs. \eqref{papaakiev} naturally extend to the free monoid $(A\cup A^{-1})^\ast$, and they are well defined also on the free group $F_A$ generated by $A$, and on the automaton group $G(\mathcal{A})$.\\
\indent It is a remarkable fact that an automaton group can be regarded in a very natural way as a group of automorphisms of the rooted regular tree of degree $|X|=k$, i.e., the rooted tree $T_k$ in which each vertex has $k$ children, via the identification of the $k^n$ vertices of the $n$-th level $L_n$ of $T_k$ with the set $X^n$ (in Fig. \ref{alberone} the first three levels of the rooted tree $T_4$ are represented). Moreover, the boundary of $T_k$ is naturally identified with the set $X^\infty$.

\begin{figure}[h] \begin{center}\begin{tiny}

\psfrag{O}{$\emptyset$}
\psfrag{0}{$0$} \psfrag{1}{$1$} \psfrag{2}{$2$} \psfrag{3}{$3$}

\psfrag{00}{$00$} \psfrag{01}{$01$}   \psfrag{02}{$02$} \psfrag{03}{$03$}
\psfrag{10}{$10$} \psfrag{11}{$11$}   \psfrag{12}{$12$} \psfrag{13}{$13$}
\psfrag{20}{$20$} \psfrag{21}{$21$}   \psfrag{22}{$22$} \psfrag{23}{$23$}
\psfrag{30}{$30$} \psfrag{31}{$31$}   \psfrag{32}{$32$} \psfrag{33}{$33$}

\includegraphics[width=0.90\textwidth]{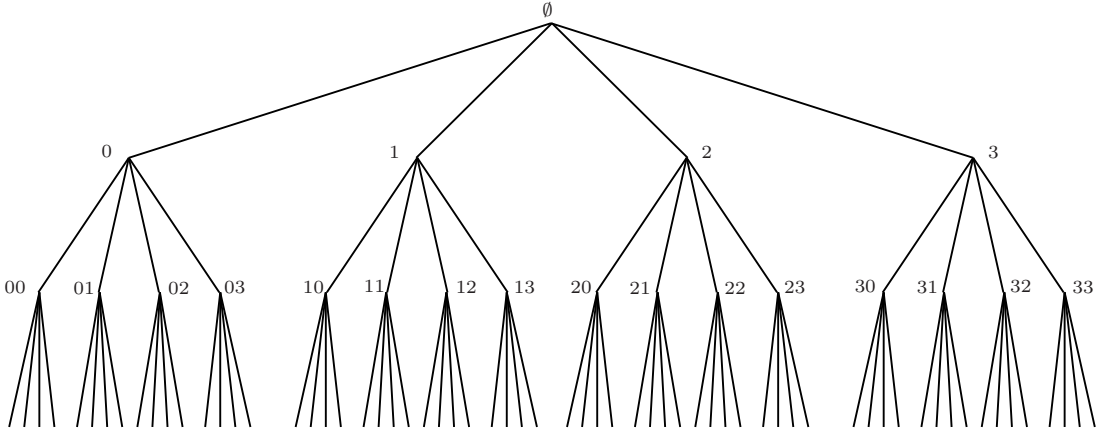} \end{tiny}
\end{center}\caption{The first three levels of the rooted tree $T_4$.} \label{alberone}
\end{figure}

Put $X=\{1,\ldots, k\}$. The action of $g\in G(\mathcal{A})$ (resp. in $S(\mathcal{A})$) on $X^\ast$ can be factorized by considering its action on $X$ and $|X|$ restrictions as follows. Let $Sym(k)$ be the symmetric group on $k$ elements. Then an element $g\in G(\mathcal{A})$ (resp. in $S(\mathcal{A})$) can be represented as
$$
g=(g_1,\ldots, g_{k})\sigma,
$$
where $g_i:=g\cdot i\in G(\mathcal{A})$  (resp. in $S(\mathcal{A})$) and $\sigma\in Sym(k)$ describes the action of $g$ on $X$. This is the \emph{self-similar representation} of $g$. In the tree interpretation, the permutation $\sigma$ corresponds to the action of $g$ on the first level $L_1$ of $T_k$, and the automorphism $g_i$ is the restriction of the action of $g$ to the subtree (isomorphic to $T_k$) rooted at the $i$-th vertex of the first level.

An important class of automata is given by the so-called bounded automata \cite{sidki}. An automaton is said to be \emph{bounded} if the sequence of numbers of distinct paths of length $n$ avoiding the sink state (along the directed edges of the Moore diagram) is bounded.
Finally, it is known that, if the automaton $\mathcal{A}$ is bounded, then the group $G(\mathcal{A})$ is amenable (see, e.g., \cite{amenability bounded}), i.e., it admits a finitely additive, left-invariant, probability measure. This concept is strictly related to that of growth. Given a finitely generated group with respect to a symmetric generating set, its growth describes how asymptotically behaves the function that counts the number of elements that can be written as a product of $n$ generators. It is known that a group of subexponential growth is amenable (see \cite{TC} for more details). A celebrated result of Gromov says that a finitely generated group is virtually nilpotent if and only if it has polynomial growth \cite{gromov}.

Let $G(\mathcal{A})$ be an automaton group. For a vertex $v\in T_k$, we denote by $\St_{G(\mathcal{A})}(v)$ the stabilizer of $v$, that is, the subgroup consisting of all elements of $G(\mathcal{A})$ fixing $v$. We also put $\St_{G(\mathcal{A})}(L_n) = \bigcap_{v\in L_n} \St_{G(\mathcal{A})}(v)$. Observe that $\St_{G(\mathcal{A})}(L_1) = \St_{G(\mathcal{A})}(X)$. \\
\indent Let $\psi: \St_{G(\mathcal{A})}(X)\to G(\mathcal{A})^k$ be the map associating with an element $g\in \St_{G(\mathcal{A})}(X)$ the $k$-tuple $(g_1,\ldots, g_{k})$ consisting of its restriction to the $k$ subtrees rooted at the first level of the tree. Then $G(\mathcal{A})$ is said to be:
\begin{enumerate}
\item \emph{fractal}, if the map $\psi$ is surjective on each factor;
\item \emph{weakly regular branch} over its subgroup $H$, if $H^k\subset\psi(H\cap \St_{G(\mathcal{A})}(X))$, where $H$ is supposed to be nontrivial;
\item \emph{contracting} if there exists a finite set $\mathcal{N}$ of $G(\mathcal{A})$ such that, for every $g\in G(\mathcal{A})$, there exists $n\in \mathbb{N}$ such
that $g\cdot v\in \mathcal{N}$ for all the words $v\in X^{\geq n}$. The minimal set $\mathcal{N}$
with this property is called nucleus.
\end{enumerate}
The reader further interested in these definitions and related properties is referred, for instance, to \cite{fractal_gr_sets}. The action of $ G(\mathcal{A})$ can be naturally extended to the boundary $X^{\infty}$. In this case, an infinite sequence $\xi\in X^{\infty}$ is said to be a singular point for the action of $G(\mathcal{A})$ if the stabilizer of $\xi$ coincides with its neighborhood stabilizer (see \cite{voro}).

The \emph{dual} of $\mathcal{A}$, denoted by $\partial \mathcal{A}$, is the automaton $(X,A,\mu,\lambda)$ having the transition $x\mapright{s|t}y$ whenever $s\mapright{x|y}t$ is a directed arrow in the Moore diagram of $\mathcal{A}$. In the rest of the paper, we will make a large use of the notion of dual automaton, and we will often identify the following maps:
\begin{itemize}
\item $s\cdot x$ in $\mathcal{A}$ is equivalent to $x\circ s$ in $\partial \mathcal{A}$;
\item $s\circ x$ in $\mathcal{A}$ is equivalent to $x\cdot s$ in $\partial \mathcal{A}$.
\end{itemize}
Clearly, we can extend such transformations as for $\mathcal{A}$ and the same equivalences for elements from $X^\ast$ and $A^\ast$ hold.

In the case $\mathcal{A}$ is invertible, one may carry on the action of $\mathcal{A}$ to the disjoint union $\mathcal{A}\sqcup \mathcal{A}^{-1}$ in the obvious way: by adding transitions of the form $t^{-1}\mapright{y|x}s^{-1}$ whenever $s\mapright{x|y}t$ in $\mathcal{A}$. The dual of $\mathcal{A}\sqcup \mathcal{A}^{-1}$ will be called \emph{enriched dual} and will be denoted by $(\partial\mathcal{A})^{-}$ (see Example \ref{ex1}). This has been already used in \cite{freeness} as a tool to determine relations in $G(\mathcal{A})$.

\section{Graph automaton groups and words} \label{sectiontre}

In \cite{articolo0} we introduced the following construction associating an invertible automaton (and so an automaton group) with a given finite graph.

\begin{defn}[Graph automaton group]\label{quartadose}
Let $G=(V,E)$ be a finite graph, with vertex set $V$ and edge set $E$. An orientation of $E$ is a function $\mathcal{E}:E\to V\times V$ with the property that either $\mathcal{E}(e)=(x,y)$ or $\mathcal{E}(e)=(y,x)$ for any $e=\{x,y\}\in E$. Let $\mathcal{A}_{G,\mathcal{E}}=(\mathcal{E}(E) \cup \{id\}, V, \lambda, \mu)$ be the automaton such that:
\begin{itemize}
\item $\mathcal{E}(E) \cup \{id\}$ is the set of states;
\item $V$ is the alphabet;
\item $\lambda: \mathcal{E}(E)\times V\to \mathcal{E}(E) \cup \{id\}$ is the restriction map such that, for each $e=(x,y)\in \mathcal{E}(E)$, one has
$$
\lambda(e,z)= \left\{
                  \begin{array}{ll}
                    e & \hbox{if } z=x \\
                    id & \hbox{if } z\neq x;
                  \end{array}
                \right.
$$
\item $\mu: \mathcal{E}(E)\times V\to V$ is the output map such that, for each $e=(x,y)\in \mathcal{E}(E)$, one has
$$
\mu(e,z) = \left\{
                  \begin{array}{ll}
                    y & \hbox{if } z=x \\
                    x & \hbox{if } z=y \\
                    z & \hbox{if } z\neq x,y.
                  \end{array}
                \right.
$$
\end{itemize}
In words, any oriented edge $e=(x,y)$ is a state of the automaton $\mathcal{A}_{G,\mathcal{E}}$ and it has just one restriction to itself  and all other restrictions to the sink $id$. Its action is nontrivial only on the letters $x$ and $y$, which are switched since $e\circ  x=y$ and $e\circ y=x$. It is easy to check that $\mathcal{A}_{G,\mathcal{E}}$ is invertible. The \emph{graph automaton group} $\mathcal{G}_{G,\mathcal{E}}$ is the automaton group generated by $\mathcal{A}_{G,\mathcal{E}}$.
\end{defn}

Note that the graph automaton group is well defined since it is independent on the orientation. In fact, the change of the orientation of an edge corresponds to consider the inverse of that generator in the group. For this reason, we will use the notation $\mathcal{G}_{G}$ for a graph automaton group. Moreover, the graph automaton group associated with the disjoint union of  graphs is isomorphic to the direct product of the corresponding graph automaton groups (see \cite[Proposition 3.5]{articolo0}): therefore, from now on we will assume $G$ to be connected.
Furthermore, it follows from the definition that the generators associated with two edges of $G=(V,E)$ commute if and only if those edges are not incident, since their action is nontrivial on disjoint subsets of $V$.
In \cite[Theorem 3.7]{articolo0} it is shown that the automaton $\mathcal{A}_{G,\mathcal{E}}$ is bounded, so that the group $\mathcal{G}_G$ is amenable; moreover, whenever $|E|\geq 2$, one has that $\mathcal{G}_G$ is a fractal group and it is weakly regular branch over its commutator subgroup $\mathcal{G}'_G$. In this paper, we will focus on graph automaton (semi)groups associated with a tree $T$. When the orientation $\mathcal{E}$ of the edges of the graph $G=(V,E)$ is fixed, we simply use the notation $\mathcal{A}_{G}$ for the automaton $\mathcal{A}_{G,\mathcal{E}}$.

\begin{example}\label{ex1}\rm
Consider the oriented star graph $S_3$ with vertex set $\{0,1,2,3\}$ and edge set $\{a,b,c\}$, and the associated automaton $\mathcal{A}_{S_3}$ depicted in Fig. \ref{unicafig}.
\begin{figure}[h]
\begin{center}
\psfrag{a}{$a$}\psfrag{b}{$b$}\psfrag{c}{$c$}\psfrag{0}{$0$}\psfrag{1}{$1$}\psfrag{2}{$2$}\psfrag{3}{$3$}
\footnotesize
\psfrag{A}{$a$}\psfrag{B}{$b$}\psfrag{C}{$c$}\psfrag{id}{$id$}\psfrag{ia}{$1|0, 2|2, 3|3$}\psfrag{ib}{$2|0, 1|1, 3|3$}\psfrag{ic}{$3|0, 1|1, 2|2$}\psfrag{ta}{$0|1$}\psfrag{tb}{$0|2$}\psfrag{tc}{$0|3$}
\includegraphics[width=0.80\textwidth]{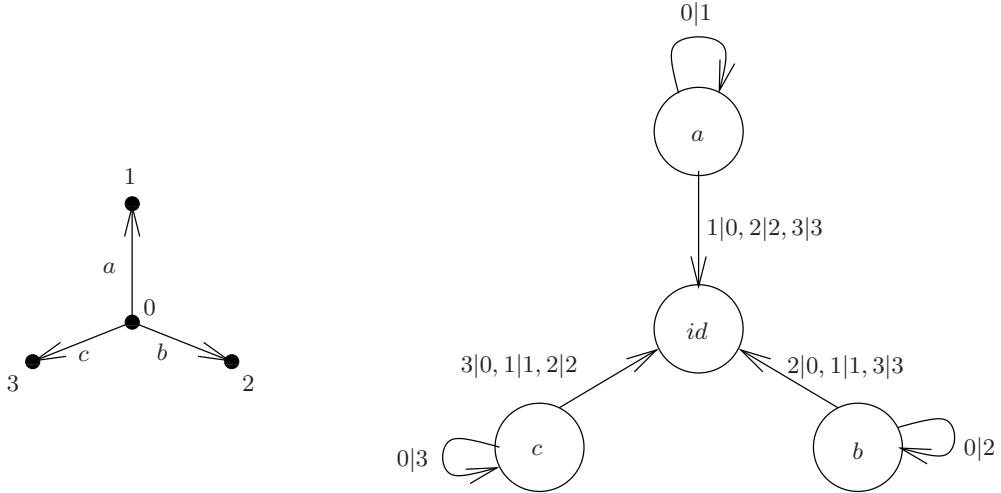}
\end{center}\caption{The oriented star graph $S_3$ and the automaton $\mathcal{A}_{S_3}$.} \label{unicafig}
\end{figure}
The dual automaton $\partial \mathcal{A}_{S_3}$ and its enriched version $(\partial \mathcal{A}_{S_3})^{-}$ are depicted in Fig. \ref{am} and Fig. \ref{aam}, respectively.
\begin{figure}[h]
\begin{center}
\footnotesize
\psfrag{0}{$0$}\psfrag{1}{$1$}\psfrag{2}{$2$}\psfrag{3}{$3$}
\psfrag{a}{$a|id$}\psfrag{b}{$b|id$}\psfrag{c}{$c|id$}\psfrag{aa}{$a|a$}\psfrag{bb}{$b|b$}\psfrag{cc}{$c|c$}
\includegraphics[width=0.50\textwidth]{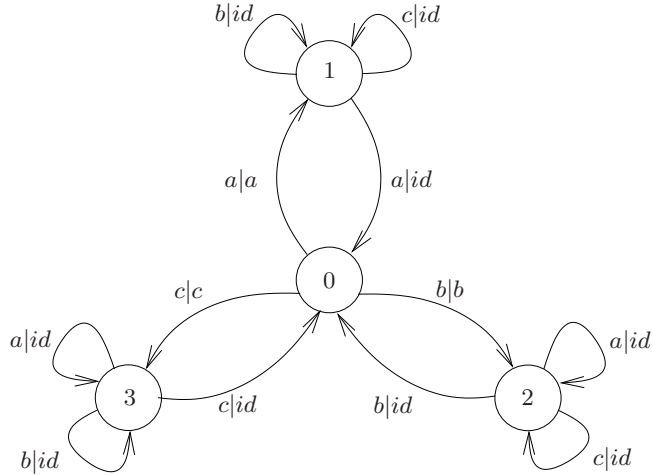}
\end{center}\caption{The dual automaton $\partial \mathcal{A}_{S_3}$.} \label{am}
\end{figure}
\begin{figure}[h]
\begin{center}
\footnotesize
\psfrag{0}{$0$}\psfrag{1}{$1$}\psfrag{2}{$2$}\psfrag{3}{$3$}
\psfrag{A}{$b|id, b^{-1}|id$}\psfrag{B}{$c|id, c^{-1}|id$}\psfrag{C}{$a|a, a^{-1}|id$}
\psfrag{D}{$a|id, a^{-1}|a^{-1}$}\psfrag{E}{$c|c, c^{-1}|id$}\psfrag{F}{$c|id, c^{-1}|c^{-1}$}
\psfrag{G}{$a|id, a^{-1}|id$}\psfrag{H}{$b|id, b^{-1}|id$}\psfrag{I}{$b|b, b^{-1}|id$}
\psfrag{L}{$b|id, b^{-1}|b^{-1}$}\psfrag{M}{$a|id, a^{-1}|id$}\psfrag{N}{$c|id, c^{-1}|id$}
\includegraphics[width=0.55\textwidth]{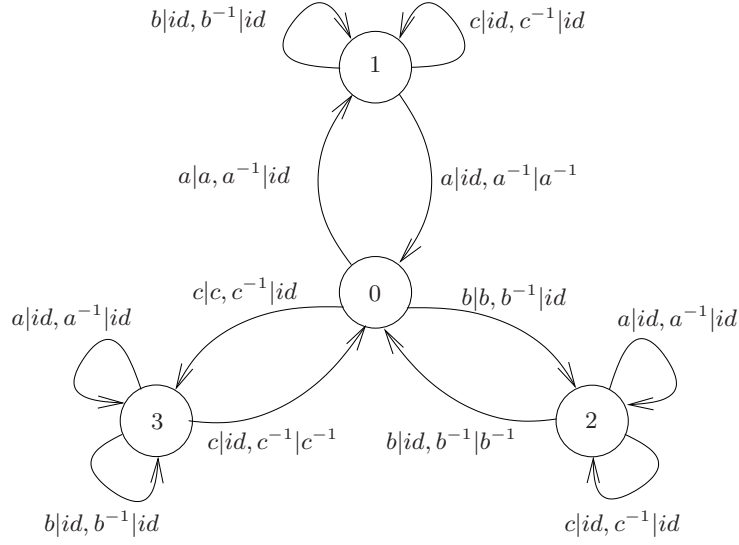}
\end{center}\caption{The enriched dual automaton $(\partial \mathcal{A}_{S_3})^{-}$.} \label{aam}
\end{figure}
\end{example}
\vspace{1cm}

In what follows we introduce some notations that are used throughout the paper. \\ Let $A$ be a finite set and let $F_A$ be the free group generated by $A$, whose elements are all the reduced words in the alphabet $\wt{A}=A\cup A^{-1}$. Let
$$
\theta: \wt{A}^{\ast}\to F_A
$$
be the canonical homomorphism associating with any word in $\wt{A}^{\ast}$ the corresponding reduced word in $F_A$. For an element $u\in F_A$, we use the notation $|u|$ for its length with respect to the generating set $A$.\\ \indent Any finitely generated group $\mathcal{G}$ can be seen as a quotient of $F_A$, for some generating set $A$, and we denote by $\pi:F_A\to \mathcal{G}$ the corresponding canonical epimorphism. In particular, any (nontrivial) relation in $\mathcal{G}$ corresponds to a nontrivial element of $F_A$ that projects onto the identity of $\mathcal{G}$.\\
\indent In our context, the generating set $A$ will be identified with the set $\mathcal{E}(E) \cup \{id\}$. Recall that the maps $\cdot$ and $\circ$ can be defined on any element $w\in F_A$, where the sink $id$ may be identified with the identity $\mathbf{1}$ of the group $F_A$.

\begin{rem}\label{rem_cruciale}\rm
The fractalness property of the graph automaton group associated with a given graph $G=(V,E)$ proven in \cite[Theorem 3.7]{articolo0} actually implies that for any $w\in F_A$ and any $v\in V^n$ there exists $w'\in F_A$, such that $w'$ stabilizes $v$ and $w'\cdot v=w$.
\end{rem}

\section{Algebraic properties of tree automaton groups}\label{sec: some properties}

The aim of the first part of this section is to describe some algebraic properties that hold in general for graph automaton groups; in the second part,
we will focus on graph automaton groups associated with graphs which are trees, that we call \emph{tree automaton groups}.

\begin{prop}\label{prop:quotientsym}
Let $G=(V,E)$ be a graph and let $\mathcal{G}_G$ be the associated graph automaton group. Then $\mathcal{G}_G/\St_{\mathcal{G}_G}(L_1)\simeq Sym(|V|)$.
\end{prop}
\begin{proof}
In order to prove our claim, we must show that the action of $\mathcal{G}_G$ restricted on $V$ recovers the whole group $Sym(|V|)$. We first observe that, if $G$ has vertex set $V=\{1,\ldots, n\}$ and the edge $e$ connects the vertices $j$ and $k$, then the projection of $e$ to $\mathcal{G}_G/\St_{\mathcal{G}_G}(L_1)$ is the transposition $(j,k)$. Since $Sym(n)$ is generated by its transpositions, if we show that any transposition $(i,j)$, with $i,j\in\{1,\ldots, n\}$  belongs to  $\mathcal{G}_G/\St_{\mathcal{G}_G}(L_1)$, then we are done. Let $i,j$ be vertices in $V$, and let $e_{i_1}^{\epsilon_1}\cdots e_{i_k}^{\epsilon_k}$, with $\epsilon_h\in\{1,-1\}$ be an oriented path from $i$ to $j$ in $G$. Then an explicit computation shows that the element
\[
e_{i_1}^{\epsilon_1}\cdots e_{i_k}^{\epsilon_k}(e_{i_1}^{\epsilon_1}\cdots e_{i_{k-1}}^{\epsilon_{k-1}})^{-1}
\]
acts on $V \equiv L_1$ like the transposition $(i,j)$. The claim follows.
\end{proof}

\begin{cor}\label{coro: non-solvable}
Let $G=(V,E)$ be a graph with $|V|=n\ge 5$. Then the graph automaton group $\mathcal{G}_G$ is not solvable.
\end{cor}
\begin{proof}
The result follows from Proposition \ref{prop:quotientsym}, since the symmetric group $Sym(n)$ is not solvable for $n\geq 5$.
\end{proof}
The following proposition holds in the general frame of automaton groups.

\begin{prop}\label{centro}
Let $\mathcal{G}$ be a fractal automaton group such that $\mathcal{G}/\St_{\mathcal{G}}(L_1)$ has trivial center. Then the center $Z(\mathcal{G})$ is trivial.
\end{prop}
\begin{proof}
Suppose by contradiction that $Z(\mathcal{G})$ contains a nontrivial element, that is, there exists a nontrivial $h\in\mathcal{G}$ such that $hg=gh$ for any $g\in \mathcal{G}$. Since $h$ is nontrivial, there exists a minimal $k$ such that $h$ acts nontrivially on the $k$-th level. This means that there exist $h'\in \mathcal{G}$ and $v\in X^{k-1}$ such that $h\circ v=v$ and $h\cdot v=h'$, where the action of $h'$ on $X$ is given by some permutation $\sigma\neq \mathbf{1}$. Since $\mathcal{G}/\St_{\mathcal{G}}(L_1)$ has trivial center, there exists $\tau\in \mathcal{G}/\St_{\mathcal{G}}(L_1)$ such that $\sigma\tau\neq \tau\sigma$. In particular, there exists $g'\in \mathcal{G}$ such that $g'$ acts like the permutation $\tau$ on $X$. Since $\mathcal{G}$ is fractal, there exists $g\in\mathcal{G}$ such that $g\in \St_{\mathcal{G}}(v)$ and $g\cdot v=g'$. Now
$$
(hg)\circ vx= g\circ (v (h'\circ x)) = v (h' g') \circ x = v\tau\sigma(x)
$$
and similarly $(gh)\circ vx=v\sigma\tau(x)$. In particular, for any $x\in X$, one must have $\tau\sigma(x)=\sigma\tau(x)$. But this is impossible.
\end{proof}

Notice that, if the graph $G=(V,E)$ has only two vertices, then the associated graph automaton group is the so-called Adding Machine, and it is isomorphic to $\mathbb{Z}$ (see \cite[Example 3.6, Part 1]{articolo0}).

\begin{cor}\label{tree}
Let $G=(V,E)$ be a graph on $n\geq 3$ vertices and let $\mathcal{G}_G$ be the corresponding graph automaton group. Then $Z(\mathcal{G}_G)$ is trivial.
\end{cor}
\begin{proof}
We know that $\mathcal{G}_G$ is fractal and in Proposition~\ref{prop:quotientsym} we have already proven that the quotient $\mathcal{G}_G/\St_{\mathcal{G}_G}(L_1)$ is isomorphic to $Sym(n)$. It is well known that $Z(Sym(n))$ is trivial for $n\geq 3$. The claim follows from Proposition \ref{centro}.
\end{proof}

From now on, we consider a graph automaton group obtained by applying the construction of Definition \ref{quartadose} to a graph $T=(V,E)$ which is a tree. We denote by $\mathcal{T}$ the associated automaton, and by  $\mathcal{G}_T$ the corresponding tree automaton group. The arboreal structure of the $T$ will be reflected on the algebraic properties of $\mathcal{G}_T$, since the dual automaton $\partial \mathcal{T}$ resembles a tree.  Indeed, if in the underlying graph of $\partial \mathcal{T}$ we forget the orientation, the loops and the multiedges, we get a simple graph which is isomorphic to the tree $T$ (see, for instance, Fig. \ref{am} for the case where $T$ is the star graph on $4$ vertices). This fact will be heavily used henceforth.

Although the state set of the automaton $\mathcal{T}$ is $\mathcal{E}(E) \cup \{id\}$, we will refer to $E$ as the generating set, since the sink $id$ is identified with the identity element of the free group $F_{\mathcal{E}(E) \cup \{id\}}$; therefore, with a slight abuse of notation, we will write $F_E$. Given an element $w\in F_E$ and a word $u\in V^\ast$, if $w\cdot u$ (or $u\circ w$ in $\partial \mathcal{T}$) contains some occurrences of $id$, we will identify such occurrences with $\mathbf{1}$.
\begin{lemma}\label{lem: loop reduces}
Let $T=(V,E)$ be a tree. Let $w\in F_E$ be a nontrivial element and  let $x\in V$ be a vertex such that $w\circ x=x$. Then $|w\cdot x|<|w|$.
\end{lemma}
\begin{proof}
Let us prove the statement by induction on the length $|w|$. The base case $|w|=1$ is trivial, since in this case in the dual $\partial\mathcal{T}$ $w$ is a loop corresponding to an edge of $T$ not adjacent to $x$, hence $w\cdot x= \mathbf{1}$. Let us assume the statement to be true for all words of length strictly smaller than $|w|$. Let $w=aw'$, with $a\in E$. We have two cases:
\begin{itemize}
\item $a\circ x=x$, and so in this case we have that $a\cdot x=\mathbf{1}$ and so the statement holds since $|w\cdot x|<|w|$.
\item $a\circ x=y$ with $y\neq x$. Without loss of generality, we can suppose that $a=(x,y)$ is an oriented edge of $T$. This implies that $w'\circ y = x$, since $w=aw'$ fixes $x$ by hypothesis. Now, if there exists a nonempty prefix $u$ of $w'$ such that $u\circ y=y$, with $u'=u\cdot y$, then, by the induction hypothesis, we have $|u'|=|u\cdot y|<|u|$. Therefore, we get
\begin{eqnarray*}
w\cdot x&=& (aw')\cdot x = (auv)\cdot x = (a\cdot x)(u\cdot (a\circ x))(v \cdot (au\circ x))\\
 &=& (a\cdot x)(u\cdot y) u''= (a\cdot x)u'u''
\end{eqnarray*}
for some word $u''$, hence $|w\cdot x|<|w|$. If there is no such a nonempty prefix $u$, then by the tree-like structure of $\partial\mathcal{T}$, necessarily the word $w$ starts with $a^2$, since $w$ is supposed to be reduced. Since $a^2\cdot x=a$, we certainly have $|w\cdot x|<|w|$ and the proof is concluded.
\end{itemize}
\end{proof}
The previous result allows us to obtain the following property about relations in a graph automaton group.

\begin{theorem}\label{relazioni}
Let $G=(V,E)$ be an (oriented) graph and let $w\in F_E$ a relation in $\mathcal{G}_G$, so that $\pi(w)=\mathbf{1}$, such that the generators occurring in $w$ constitute a subset $\{e_1,\ldots,e_t\}$ of edges whose induced subgraph in $G$ does not contain any cycle. Then the sum of the exponents of each of the $e_i$'s in $w$ is zero.
\end{theorem}
\begin{proof}
First of all we notice that the statement is invariant under the orientation of the $e_i$'s. Now remark that generators belonging to disjoint subgraphs of $G$ act on disjoint alphabets and so they give rise to a trivial commutator, which is a relation that satisfies our claim.
This implies that we can restrict our attention to relations $w$ involving edges whose induced subgraph in $G$ is a tree $T$ with edge set $E=\{e_1,\ldots,e_t\}$.\\
\indent Let $w\in F_E$ be such a relation of $\mathcal{G}_G$ and let $w=(w_1,\ldots, w_n)$ be its self-similar representation, where $n$ is the cardinality of the vertex set $V$ of $G$.
Let us prove the statement by induction on $|w|$.
It is easy to show that the shortest relations are given by commutators of nonadjacent edges, which have the claimed property.
If $w$ is a relation of length greater than $4$, then  for any $x\in V$ one has $w\circ x=x$ and so by Lemma \ref{lem: loop reduces} $|w_x|=|w\cdot x|<|w|$. Since each of the $w_x$'s is a relation, being a restriction of a relation, by inductive hypothesis each of the $w_x$'s has the property that  the sum of the exponents of each of the $e_i$'s is zero. Finally, observe that in each graph automaton group, any occurrence of $e_i^{\pm 1}$ appearing in $w$, appears in exactly one restriction $w_y$, for some $y$. This concludes the proof.
\end{proof}
Recall that the set of words of $F_E$ such that the sum of the exponents of each generator is zero coincides with the commutator subgroup $F_E'$. Then, by Theorem \ref{relazioni}, the sum of the exponents of each of the $e_i$'s in any relation of a tree automaton group $\mathcal{G}_T$ is zero, and so every relation in $\mathcal{G}_T$ belongs to $F_E'$. The following corollary holds.
\begin{cor}\label{derivato}
Let $T$ be a tree with edge set $E=\{e_1,\ldots,e_t\}$. Then
$$
\mathcal{G}_T/\mathcal{G}_T'\simeq \mathbb{Z}^{t}.
$$
\end{cor}
\begin{proof}
Let $\pi: F_E \to \mathcal{G}_T$. We have to show that $ \pi(e_1^{i_1}\cdots e_t^{i_t})\not\in \mathcal{G}_T'$ for any $(i_1, \ldots, i_t)\neq (0,\ldots, 0)$. Suppose by contradiction that there exists $(i_1, \ldots, i_t)\neq (0,\ldots, 0)$ such that $ \pi(e_1^{i_1}\cdots e_t^{i_t})=g\in \mathcal{G}_T'$. Let $w\in F_E'$ be a word representing $g$. The word $u=w^{-1} e_1^{i_1}\cdots e_t^{i_t}$ is a relation in $\mathcal{G}_T$. Thus, by Corollary~\ref{cor: zero sum}, in $u$ the sum of the exponents of each of the $e_i$'s is equal to zero. Now, since $w\in F_E'$ has also the same property, we deduce that $(i_1,\ldots, i_t)=(0,\ldots, 0)$, a contradiction.
\end{proof}
Moreover, Theorem \ref{relazioni} allows us to easily describe a family of nontorsion elements in $\mathcal{G}_T$.

\begin{cor}\label{cor: zero sum}
Let $T=(V,E)$ be a tree. Then any  word $w\in F_E$ such that the sum of the exponents of some generator appearing in $w$ is nonzero is not torsion.
\end{cor}
\begin{proof}
The last statement follows from the observation that the sum of the exponents in $w^n$ is $n$ times the sum of the exponents in $w$.
\end{proof}

\section{Tree automaton semigroups}\label{sec: semigroups}
In this section we study the automaton semigroup associated with a tree.

\begin{defn}[Graph automaton semigroup]\label{quintadose}
Let $G=(V,E)$ be as in Definition \ref{quartadose} with an orientation $\mathcal{E}$ of its edges. Let $\mathcal{A}_{G,\mathcal{E}}$ be the associated automaton. The \emph{graph automaton semigroup} $\mathcal{S}_{G,\mathcal{E}}$ is the semigroup generated by the automaton $\mathcal{A}_{G,\mathcal{E}}$.
\end{defn}
While for graph automaton groups the orientation of the graph is irrelevant, different orientations of the graph may give rise to non isomorphic semigroups, as we will later observe in Remark \ref{interesse}. In this section we show that, in the case of a tree, the orientation is irrelevant: different orientations give rise to semigroups that are all isomorphic to the same partially commutative monoid.
When the graph $G$ is fixed, we will denote the corresponding semigroup by $\mathcal{S}_{\mathcal{E}}$ omitting the index $G$. Moreover, in order to simplify the notation, we always think the edge set $E$ endowed with an orientation $\mathcal{E}$.       \\

Notice that the generating automaton contains the sink state $id$, which corresponds to the identity of the generated semigroup. For this reason, we also include $id$ in the generating set of $\mathcal{S}_{\mathcal{E}}$. We put $A=E\cup \{id\}$ in order to denote such a generating set. Remark that, in this way, the dual automaton $\partial \mathcal{A}_{G,\mathcal{E}}$ everywhere contains loops of the form $x\vlongfr{id}{id}x$.

Let us fix a tree $T=(V,E)$ and let $\mathcal{E}$ be an orientation of its edges. We have already remarked that two generators commute if and only if they correspond to two edges which are not incident in $T$. Actually, these are the only nontrivial relations in $\mathcal{S}_{\mathcal{E}}$. In fact, we will show (see Theorem \ref{theo: submonoid is a trace monoid}) that the semigroup $\mathcal{S}_{\mathcal{E}}$ is isomorphic to the partially commutative monoid with presentation
\[
\la A\mid id=\mathbf{1}, ab=ba\mbox{ if }\{a,b\}\mbox{ is an edge in }L(T)^c\ra,
\]
where $L(T)^c$ is the graph obtained by complementing the line graph of $T$, i.e., two vertices $a,b$ are connected in $L(T)^c$ if the corresponding edges $a,b\in E$ of $T$ are not incident, and in this case $a$ and $b$ commute in $\mathcal{S}_{\mathcal{E}}$.

\begin{example}\rm
In Fig. \ref{figlinegraph} an example of an oriented tree $T$, together with its line graph $L(T)$ and its complement $L(T)^c$, is represented.
\begin{figure}[h]
\begin{center}
\footnotesize
\psfrag{1}{$1$}\psfrag{2}{$2$}\psfrag{3}{$3$}\psfrag{4}{$4$}\psfrag{5}{$5$}\psfrag{6}{$6$}
\psfrag{e1}{$e_1$}\psfrag{e2}{$e_2$}\psfrag{e3}{$e_3$}\psfrag{e4}{$e_4$} \psfrag{e5}{$e_5$}
\includegraphics[width=0.80\textwidth]{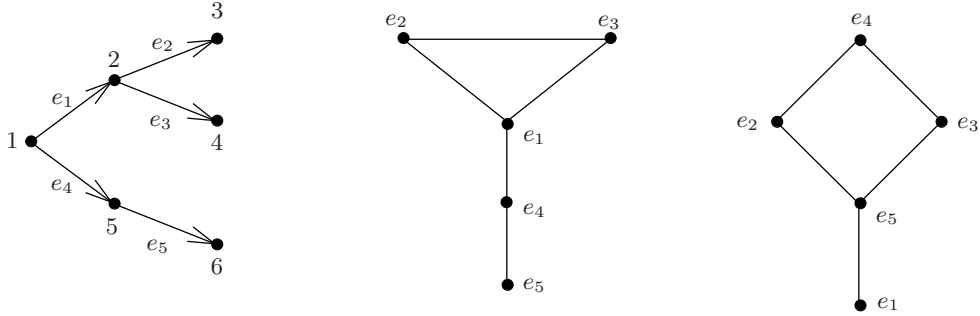}
\end{center}\caption{An oriented tree $T$ and the graphs $L(T)$ and $L(T)^c$.} \label{figlinegraph}
\end{figure}
\end{example}

For each word $u=a_1\ldots a_n\in A^*$ and vertex $x\in V$ we denote by $P(x,u)$ the path
$$
x\vvlongfr{a_1}{x\circ a_1}x_1\vvlongfr{a_2}{x_1\circ a_2}\ldots\ldots \vvlongfr{a_{n-1}}{x_{n-2}\circ a_{n-1}} x_{n-1}\vvlongfr{a_n}{x_{n-1}\circ a_{n}}x_{n}
$$
in $\partial \mathcal{T}$ starting at $x$ induced by the word $u$, while we denote by $p(x,u)$ the sequence of vertices of $P(x,u)$ without consecutive repetitions. For instance, consider the dual automaton associated with the oriented tree depicted in Fig. \ref{figlinegraph}. If we take the word $w=e_2e_1e_1e_4$, then:
$$
P(1,w)=1\vlongfr{e_2}{id}1\vlongfr{e_1}{e_1}2\vlongfr{e_1}{id}1\vlongfr{e_4}{e_4}5 \qquad \qquad p(1,w)=1215.
$$
We consider the following set
\[
A_x=\{a_i\in A:\mbox{ such that }x\vlongfr{a_i}{id} x\mbox{ are loops in }\partial  \mathcal{T}\}
\]
which contains $id$ together with the set of edges of $T$ that are not incident to $x$. We denote by $\varepsilon: A^*\to E^*$ the morphism that erases the identity state $id$.
The following lemma follows from the definition of $A_x$.

\begin{lemma}\label{lem: not incident}
Consider a word $u\in A^*$ such that $p(x,u)=xy\ldots$, with $x\neq y$. Then we may factorize $u=u'\alpha u''$, where $u'\in A_x^*$ and $\alpha\in E$ is the edge connecting $x$ and $y$.
\end{lemma}

We say that a word $v\in A^*$ has a \emph{noose} at $x$ if $v$ contains a nontrivial factor $u = a_1 \ldots a_k$ such that the path $P(x,v)$ in $\partial \mathcal{T}$ contains a subpath
\[
x\vlongfr{u}{u'}x \quad = \quad x_1\vlongfr{a_1}{b_1}x_2\vlongfr{a_2}{b_2}\ldots \vlongfr{a_k}{b_k}x_{k+1}
\]
where the only vertices equal to $x$ are the first one $x_1$ and the last one $x_{k+1}$ (see Fig. \ref{fignoose}). Note that by the tree-like structure of $\partial \mathcal{T}$ it must be $a_1=a_k$.
\begin{figure}[h]
\begin{center}
\footnotesize
\psfrag{x}{$x$}\psfrag{x2}{$x_2$}\psfrag{a1}{$a_1|b_1$}\psfrag{a2}{$a_2|b_2$}\psfrag{ak-1}{$a_{k-1}|b_{k-1}$}\psfrag{ak}{$a_k|b_k$}
\psfrag{e1}{$e_1$}\psfrag{e2}{$e_2$}\psfrag{e3}{$e_3$}\psfrag{e4}{$e_4$} \psfrag{e5}{$e_5$}
\includegraphics[width=0.70\textwidth]{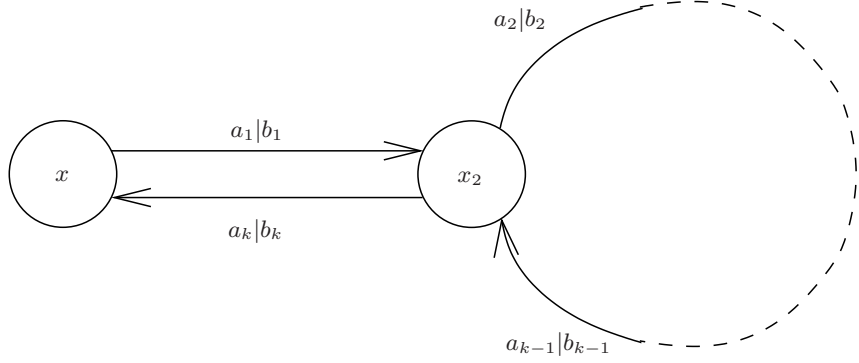}
\end{center}\caption{A noose in $\partial \mathcal{T}$.} \label{fignoose}
\end{figure}

\begin{lemma}\label{lem: noose}
Let $T=(V,E)$ be a tree and let $\mathcal{E}$ be an orientation of $E$. Let $\alpha=(x,y)$ be an oriented edge from $x$ to $y$. Suppose that $u=\alpha w\alpha$ is a word in $A^\ast$ such that
\[
x\vlongfr{u}{u'}x \quad = \quad x\vlongfr{\alpha}{\alpha}y\vlongfr{w}{w'}y\vlongfr{\alpha}{id}x
\]
is a path in $\partial \mathcal{T}$ and $w$ does not contain any occurrence of $\alpha$. Then $\varepsilon(u')=\varepsilon(\alpha w'id)=\alpha\varepsilon(w')$ is such that $\varepsilon(w')$ is a word formed by edges that are not incident to $x$.\\
\indent Analogously, if $u=\alpha w\alpha$, with $\alpha=(y,x)$ is a word in $A^\ast$ such that
\[
x\vlongfr{u}{u''}x\quad = \quad x\vlongfr{\alpha}{id}y\vlongfr{w}{w'}y\vlongfr{\alpha}{\alpha}x
\]
is a path in $\partial\mathcal{T}$ and $w$ does not contain any occurrence of $\alpha$, then $\varepsilon(u'')=\varepsilon(idw'\alpha ) = \varepsilon(w')\alpha$ is such that $\varepsilon(w')$ is a word formed by edges that are not incident to $x$.
\end{lemma}
\begin{proof}
The statement follows by observing that the edges forming the word $\varepsilon(w')$ are edges that belong to the connected component containing the vertex $y$ and obtained from $T$ by cutting the edge $\alpha$. Since $T$ is a tree, these edges are clearly not incident to $x$.
\end{proof}

\begin{lemma}\label{lem: forbitten edges}
Let $\alpha=(x,y)$ be an oriented edge in $E$, and let $u\in A^*$ be a word such that $p(x,x\circ u)=xy\ldots$. Then we may factorize
\begin{equation}\label{factor}
u=w\alpha w', \qquad \mbox{with } w\in A_x^*, \,w'\in A^*.
\end{equation}
\end{lemma}
\begin{proof}
By Lemma~\ref{lem: not incident} we have that  $x\circ u=v\alpha v'$, with $v\in A_x^*$. Let $w$ be the prefix of $u$ such that $x\circ w= v$. If $w\in A_x^*$, then the factorization in Eq. \eqref{factor} is obtained. Suppose now that $w\notin A_x^*$, so that $w$ contains a letter that is an edge incident to $x$. Let $\gamma$ be the first letter occurring in $w$ having this property.
We want to prove that it must be $\gamma = \alpha$, and then the factorization in Eq. \eqref{factor} still holds.\\
By contradiction, assume that either $\gamma=(x,z)$ or $\gamma=(z,x)$, with $z\neq y$, depending on the orientation. Let $u=s\gamma s'\alpha w'$ be a word, with $s\in A_x^*$, $s\gamma s'=w$, and $w'$ is such that $v' = y \circ w'$. Since $v\alpha=x\circ w\alpha$, then it must be $x\cdot w=x$. By the tree-like structure of $\partial \mathcal{T}$ it happens that $w$  must have a (first) noose $\gamma t \gamma$, where $t$ is a cyclic path on $z$. It follows from Lemma \ref{lem: noose} that, depending on the orientation of $\gamma$, one has either $\varepsilon(x\circ \gamma t \gamma)=t'\gamma$ or $\varepsilon(x\circ \gamma t \gamma)=\gamma t'$. In both cases $p(x,x\circ u) = xz\ldots \neq xy\ldots$, a contradiction.
\end{proof}

\begin{prop}\label{lem: initial path}
Let $u\equiv v$ be a relation of $\mathcal{S}_{\mathcal{E}}$ and let $x\in V$. Then:
\begin{itemize}
\item either $p(x,u)=p(x,v)=x$;
\item or $p(x,u)$ and $p(x,v)$ both have length at least $2$ and start with the same two vertices, i.e., $p(x,u)=xy\ldots$ if and only if $p(x,v)=xy\ldots$.
\end{itemize}
\end{prop}
\begin{proof}
We first show that, given $x\in V$, it is impossible that $p(x,u)=x$ and $p(x,v)=xy\ldots$, for some $y\in V$. In fact, in this case one has $u\in A_x^*$ and, on the other hand, $v$ must contain an occurence $\alpha$, where $\alpha$ is the edge joining $x$ and $y$. Since $u\equiv v$ is a semigroup relation, then $w:=vu^{-1}$ is a group relation and, by Corollary \ref{cor: zero sum}, the exponent sum of the generator $\alpha$ in $w$ must be zero. This is in contradiction with the fact that $u$ does not contain $\alpha$ and $v$ only contains positive powers of $\alpha$. \\
It follows from the previous argument that we may have either $p(x,u)=p(x,v)=x$ and there is nothing to prove, or both $p(x,u)$ and $p(x,v)$ have length greater than $1$. Hence we have to show that $p(x,u)=xy\ldots$ if and only if $p(x,v)=xy\ldots$, for some $y\in V$.\\
Let us prove the last claim by induction on $L(u,v)=|\varepsilon(u)|+|\varepsilon(v)|$. One may check by a direct computation that the minimal relation $u=v$ verifies $L(u,v)=4$, with $u=ab$ and $v=ba$, where $a,b$ correspond to nonadjacent edges, and only one between $a$ and $b$ is incident to $x$. In this case, the claim follows.\\
\indent Suppose now by contradiction that there is a vertex $x\in V$ such that $p(x,u)=xy\ldots$ and $p(x,v)=xz\ldots$ with $y\neq z$. Since $u\equiv v$ is a relation, it must be $x\cdot u=x\cdot v$. Hence by the tree-like structure of $\partial \mathcal{T}$ we have that $u$ or $v$ has a noose at $x$, let us suppose that $u$ has such a noose. Let $s=\alpha w\alpha$ be the factor of $u$ that is a noose and let us suppose that it is the first, in the sense that if $u=u'su''$, then no factor of $u's$ is a noose at $x$. Since $p(x,u)=xy\ldots$, we have that $u'$ is a word on $A_x^\ast$, so that $\varepsilon(x\circ u')$ is the empty word, and $\alpha$ is either equal to $(x,y)$ or to $(y,x)$. Now, by Lemma~\ref{lem: noose}, $\varepsilon(x\circ u)= \varepsilon(x\circ s)$ is either equal to $\varepsilon(\alpha w'u'')$ or to $\varepsilon(w'\alpha u'')$ (depending on the orientation of $\alpha$) where $\varepsilon(w')$ is a word formed by edges that are not incident to $x$. Hence, in both cases we have $p(x,x\circ u)=xy\ldots $. Since $x\circ u\equiv x\circ v$ is also a relation and $L(x\circ u, x\circ v)<L(u,v)$ because $u$ contains a noose, we can apply induction to deduce that $p(x,x\circ v)=p(x,x\circ u)=xy\ldots $. Therefore, by Lemma~\ref{lem: forbitten edges} we may factorize $v$ as $v'\alpha v''$, with $v'\in A_x^*$ and $v''\in A^*$. Hence, $p(x,v)=xy\ldots $, a contradiction.
\end{proof}
Let $u,v \in A^\ast$. We write $u\rightarrow_c v$ if:
\begin{itemize}
\item either $\varepsilon(u) = \varepsilon(v)$;
\item or there exist $a,b \in E$ such that $u=u'abu''$ and $v=u'bau''$, where $a,b$ are non incident edges of $E$.
\end{itemize}
We write $u\rightarrow^*_c v$ if there is a sequence $u=v_1,\ldots, v_n=v$ of words such that
\[v_1\rightarrow_c v_2\rightarrow_c \ldots  \rightarrow_c v_{n-1}\rightarrow_c v_n.\]
Observe that the relation $\rightarrow^*_c$ is an equivalence relation.

\begin{prop}\label{lemmaunico}
Let $u,v\in A^*$. Then $u\equiv v$ is a relation in $\mathcal{S}_{\mathcal{E}}$, with respect to the generating set $A$, if and only if $u\rightarrow^*_c v$.
\end{prop}
\begin{proof}
If $u\rightarrow^*_c v$, then it is straightforward to verify that $u\equiv v$ is a relation in $\mathcal{S}_{\mathcal{E}}$.

Let us prove the converse implication. Since, by definition, we have $u \rightarrow^*_c \varepsilon(u)$ and  $v \rightarrow^*_c \varepsilon(v)$, and $\varepsilon(u)\equiv \varepsilon(v)$ by hypothesis, we may assume $u,v\in E^\ast$. By Theorem \ref{relazioni} $u \equiv id$ is a relation if and only if $u=id$. Thus we can suppose, without loss of generality, that $|u|\geq 1$ and $|v|\geq 1$.\\
\indent Let us proceed by induction on $|u|+|v|$. If $|u|+|v|=2$, then it must be $u,v\in E$, and so $u\equiv v$ if and only if $u=v$ by definition of graph automaton semigroup. Let  $u=a\overline{u}$, $v=b\overline{v}$ with $a,b\in E$. We may assume that $a\neq b$ since, if $u,v$ start with the same letter $a$, then by the cancellativity of $\mathcal{S}_{\mathcal{E}}$ (it embeds into the corresponding automaton group) we may apply induction on the shorter relation $\overline{u}\equiv \overline{v}$.\\
\indent We have two cases to consider: either $a,b$ are edges in $E$ sharing a vertex, or they commute. By Proposition~\ref{lem: initial path} we may exclude the first case. Indeed, if $a,b$ share a common vertex $x$, then if $y$ is the other vertex of the oriented edge $a$, and $z\neq y$ the other vertex of the oriented edge $b$, then we would have $p(x,u)=xy\ldots $ and $p(x,v)=xz\ldots$, which is in contradiction with Proposition~\ref{lem: initial path}.\\
\indent Thus, we may suppose that $a,b$ commute and put $a=(x,y)$, $b=(x',y')$. By Proposition~\ref{lem: initial path} we have that $p(x,u)=xy\ldots$ and $p(x,v)=xy\ldots$, hence by Lemma~\ref{lem: not incident} we deduce that $v=v'av''$ with $v'\in A_x^\ast$. Even more, we claim that $v'$ contains only elements that commute with $a$. Indeed, suppose by contradiction that $v'$ contains a letter $c$ that is an edge incident to $a$, we also assume that no letter in $v'$ preceding $c$ is incident to $a$, i.e., $v'=wcw'$ where $w$ does not contain any generator incident to $a$. Let $c=(y,z)$ or $c=(z,y)$, with $z\neq x$. Since $w$ does not contain any generator that is incident to $a$, and thus to $y$, we conclude that $p(y, v)=yz\ldots$. On the other hand, since $u$ starts with the letter $a$, we have $p(y,u)=yx\ldots$, in contradiction with Proposition~\ref{lem: initial path}. Thus, $v=v'av''$ where $v'$ is a word containing only elements that commute with $a$. Therefore $v\rightarrow^*_c av'v''$ holds and so $v\equiv av'v''$ is a relation by the first part of the proof. Therefore, since $u\equiv v\equiv av'v''$, we conclude  that $ u\equiv av'v''$ is a relation. Now, since both $u=a\overline{u}$ and  $av'v''$ start with the letter $a$, we may cancel out this letter to obtain a shorter relation $\overline{u}\equiv v'v''$. By the induction hypothesis we get that $\overline{u}\rightarrow_c^* v'v''$, from which we conclude that $u=a\overline{u}\rightarrow_c^* av'v''\rightarrow_c^*v'av''=v$ and so $u\rightarrow_c^* v$.
\end{proof}

From the previous proposition it follows the main result of this section.

\begin{theorem}\label{theo: submonoid is a trace monoid}
Let $T=(V,E)$ be a tree and let $\mathcal{E}$ be an orientation of $E$. Then, the automaton semigroup $\mathcal{S}_{\mathcal{E}}$ is isomorphic to the partially commutative monoid with presentation
\[
\la E \cup \{id\} \mid id=\mathbf{1}, ab=ba\mbox{ if }\{a,b\}\mbox{ is an edge in }L(T)^c\ra.
\]
\end{theorem}

We recall that the \emph{join} of two graphs $G$ and $H$ is the graph denoted by $G+H$ obtained by taking two disjoint copies of $G$ and $H$, and connecting each vertex of $G$ to each vertex of $H$. In the spirit of \cite[Proposition~3.5]{articolo0}, we obtain the following result.
\begin{cor}
Let $F=(V,E)$ be a forest formed by the trees $T_1, \ldots, T_m$, and let $\mathcal{E}$ be an orientation of this forest. If $\mathcal{S}_1, \ldots, \mathcal{S}_{m}$ are the corresponding partially commutative monoids associated with the trees $T_i$'s, then $\mathcal{S}_{\mathcal{E}}\simeq \mathcal{S}_1\times \cdots \times\mathcal{S}_{m}$. In particular $\mathcal{S}_{\mathcal{E}}$ has the following presentation:
\[
\la E \cup \{id\}\mid id=\mathbf{1}, ab=ba\mbox{ if }\{a,b\}\mbox{ is an edge in }L(T_1)^c + \cdots + L(T_m)^c\ra.
\]
\end{cor}
The importance of Theorem \ref{theo: submonoid is a trace monoid} is highlighted by the next proposition and the following remark, which show that when the graph $G=(V,E)$ is not a tree, the semigroup is not independent of the orientation.

\begin{prop}
Let $\mathcal{E}$ be an orientation of the graph $G=(V,E)$ which does not produce an oriented cycle. Then in the semigroup $\mathcal{S}_{G,\mathcal{E}}$ there are no relations of the form $w=id$, $w\in E^*$.
\end{prop}
\begin{proof}
Suppose that the set $P=\{w\in E^*: w=id\mbox{ in }\mathcal{S}_{G,\mathcal{E}}\}$ of relations is nonempty, and let $u=a_1\ldots a_k\in P$ be a word with minimal length. In $\partial \mathcal{A}_{G,\mathcal{E}}$ consider the transition $x_1\vlongfr{a_1}{a_1}x_1'$. Since $\mathcal{E}$ does not give rise to any oriented cycle, and since there is a loop $x_1\vlongfr{u}{u'}x_1$ with $u'=x_1\circ u$, we deduce that $u'$ contains at least one occurrence of the identity $id$. Now, observe that it must be also $u'=id$ in $\mathcal{S}_{G,\mathcal{E}}$, and so $\varepsilon(u')$ is a shorter word in $P$. This situation may occur only if $u'\in\{id\}^*$. However, since we have the transition $x_1\vlongfr{a_1}{a_1}x_1'$, the word $u'$ starts with $a_1$, which is not the identity, a contradiction. Therefore we conclude that $P$ is empty.
\end{proof}

\begin{rem}\label{interesse}\rm
Take a graph $G=(V,E)$ which is not a tree, and consider an orientation $\mathcal{E}_1$ without oriented cycles, and a second orientation $\mathcal{E}_2$ containing an oriented cycle $w=a_1\ldots a_k$. We claim that $\mathcal{S}_{G,\mathcal{E}_1}$ and $\mathcal{S}_{G,\mathcal{E}_2}$ are not isomorphic. By contradiction, suppose that there is an isomorphism  $\varphi:\mathcal{S}_{G,\mathcal{E}_2}\to \mathcal{S}_{G,\mathcal{E}_1}$. By \cite[Theorem 3.7]{articolo0}, we have $w^{k-1}=id$, so that $w$ is a nontrivial torsion element in $\mathcal{S}_{G,\mathcal{E}_2}$, and so $(\varphi(w))^{k-1}=id$ in $\mathcal{S}_{G,\mathcal{E}_1}$, which implies that $\varphi(w)$  is also a nontrivial torsion element in $\mathcal{S}_{G,\mathcal{E}_1}$. This contradicts the previous proposition.
\end{rem}

\section{Reducible automata and poly-context-free groups}\label{sec: a general structure theorem}

In this section, we highlight an interesting connection between a class of automaton groups (generated by automata that we call reducible, see Definition \ref{def: reducible}) containing tree automaton groups and the class of the so-called poly-context-free groups. We need some preparation.

The word problem of a group $\mathcal{G}$ finitely generated by the set $A$ consists of the set of all words in $\wt{A}^{*}$ that represent the identity of $\mathcal{G}$. For this reason, the word problem can be described as a language, i.e., a set
of words over some finite alphabet.
Let $\la A|\mathcal{R}\ra$ be the standard presentation of the group $\mathcal{G}$. We have that $\mathcal{G}\simeq F_A/N$ where $N$ is the normal closure of the defining relations $\mathcal{R}$. The normal subgroup $N$ of $F_A$ is also denoted by $WP(\mathcal{G}:A;\mathcal{R})$ in the literature \cite{asimo}.
One of the most interesting problems in this setting is the algebraic characterization of a group $\mathcal{G}$ in terms of the language theoretic properties of  $WP(\mathcal{G}:A;\mathcal{R})$.
For example, Anisimov proved that a group is finite if and only if its word problem is a regular language \cite{asimo}.
A very celebrated result in this context is the classification of the groups with context-free word problem by M\"{u}ller and
Schupp, which states that a finitely generated group has context-free
word problem if and only if it is virtually free \cite{muahu, muahu2}. We recall that a regular language is a language recognized by a finite automaton, i.e., the set of words labeling all paths from an initial state to a given set of final states of the automaton. Moreover, a (deterministic) context-free language is a language generated by a context-free grammar. It is well known that, differently from the regular case, the intersection of (deterministic) context-free languages is not, in general, a (deterministic) context-free language \cite{hop}.  For this reason, it is natural to consider the closure of (deterministic) context-free languages under intersection. The intersection of $k$ (deterministic) context-free languages is said a (deterministic) $k$-context-free language. A language is (deterministic) poly-context-free if it is (deterministic) $k$-context-free for some $k\in\mathbb{N}$ \cite{Tara}.  A group is called $k$-context-free if its word problem is a (deterministic) $k$-context-free language.  A group whose word problem is a (deterministic) poly-context-free language is called poly-context-free group. Such groups have been introduced in \cite{Tara} and they have been further explored in \cite{TCS}, where it is proved that such groups have a (deterministic) multipass word problem.\\

\indent From now on we consider an invertible automaton $\mathcal{A}=(A,X,\lambda, \mu)$ and the generated group  $G(\mathcal{A})$.
Henceforth it is important to distinguish between elements in $\wt{A}^*$ and elements of the free group $F_A$, so we have to consider the use of the following two canonical morphisms:
$$
\theta:\wt{A}^*\to F_A  \qquad \qquad  \omega:\wt{A}^*\to G(\mathcal{A}).
$$
We first introduce a class of automaton groups which includes tree automaton groups. The inspiring property is the one stated in Lemma~\ref{lem: loop reduces}, which we may slightly generalize as follows.

\begin{defn}[Reducible automaton]\label{def: reducible}
An invertible automaton $\mathcal{A}=(A,X,\lambda, \mu)$ is reducible if, for any $w\in F_A$ such that $\mu(w,x)=x$ for some $x\in X$, there is an integer $n$ such that $|\lambda(w,u)|<|w|$ for all $u\in X^{\ge n}$.
\end{defn}

Observe that $A$ may contain a sink state $id$: if this is the case, we identify it with the identity $\mathbf{1}$ of $F_A$.\\
\indent It is easy to show, by applying Lemma~2.2 of \cite{Lavoro francescano}, that the class of automaton groups generated by reducible automata is contained in the class of automaton groups that do not have singular points. Moreover, the following proposition holds.
\begin{prop}
A reducible automaton $\mathcal{A}$ generates a contracting group.
\end{prop}
\begin{proof}
Let $G(\mathcal{A})$ be the group generated by the automaton $\mathcal{A}$ and let $\pi: F_A \to G(\mathcal{A})$ be the canonical map. Consider the set
\[
\mathcal{N}=\{\pi(w) \mid w\in F_A, |w|\leq m\},
\]
where $m=|X|$. This is clearly a finite subset of $G(\mathcal{A})$ and we claim that this is the nucleus of $G(\mathcal{A})$. Indeed, let $h\not\in \mathcal{N}$ such that $h=\pi(w)$, for some $w\in F_A$ with $|w|>m$. Then, for any $x\in X$, in $(\partial \mathcal{A})^-$ there is a path
\[
x\vlongfr{w_1}{w_1'}y\vlongfr{w_2}{w_2'}y\vlongfr{w_3}{w_3'}z
\]
with $w=w_1w_2w_3$, since $m=|X|$. Now, by Definition~\ref{def: reducible} applied to $w_2$ we know that for a sufficiently large $n$ we have that $|w_2\cdot u|<|w_2|$ for all $u\in X^{\ge n}$. Thus, by iterating this argument, we may find an integer $N$ such that $|w\cdot v|\le m$ for all $v\in X^{\ge N}$, i.e., $\pi(w)\cdot v = h\cdot v\in \mathcal{N}$ for $v\in X^{\ge N}$.
\end{proof}
In conclusion, the class of groups defined by reducible automata is contained in the intersection of the class of contracting groups with the class of automaton groups without singular points. There are several groups generated by reducible automata. Lemma~\ref{lem: loop reduces} implies that tree automaton groups are generated by reducible automata. But also the famous Basilica group belongs to this class: this can be directly checked by showing that any element in the Basilica group can be reduced by restriction using a word of length at most two. A similar argument works also for a generalization of this group recently introduced in \cite{noce}. At the end of this section, we will give a general strategy to construct infinitely many reducible automata.
\\
\indent There is an interesting connection between groups defined by reducible automata and poly-context-free groups. Indeed, we will show (see Theorem \ref{thmastrazeneca}) that groups generated by reducible automata are direct limits of deterministic poly-context-free groups. Brough has conjectured in \cite{Tara} that the class of finitely generated poly-context-free groups coincides with the class of (finitely generated) groups which are virtually a finitely generated subgroup of a direct product of free groups. In this section, we partially support this conjecture by showing that the direct system of poly-context-free groups whose limit is the group associated with a reducible automaton group is formed by groups that are virtually a finitely generated subgroup of a direct product of free groups (see Proposition \ref{prop: k-fragile group structure}).\\
\indent We are going to associate with a reducible automaton $\mathcal{A}$ a direct system of groups by showing that $WP(G(\mathcal{A}):A;\mathcal{R})$ may be decomposed into a union of deterministic poly-context-free languages.
\\
\indent For each integer $k >0$, let $\St_{G(\mathcal{A})}(L_k)$ be the stabilizer of the $k$-th level of the rooted tree of degree $|X|$, identified with $X^k$. The \emph{language of $k$-fragile words} is defined as
\begin{equation}\label{defSk}
\mathcal{F}_k=\{w\in\omega^{-1}(\St_{G(\mathcal{A})}(L_k)): \theta(w\cdot u)=\mathbf{1}\mbox{ for all }u\in X^{\ge k} \}.
\end{equation}
The name is inspired by \cite{fragile, fragili} and refers to the existence of relations that eventually become trivial after certain restrictions. Note that $\mathcal{F}_k\subseteq \mathcal{F}_h$ if $h\ge k$.

We recall that a (symmetric) Dyck language over an alphabet $A$ is the word problem of the free group $F_A$ (see \cite{cho} for more details). We have the following proposition.

\begin{prop}\label{prop: decomposition by k-fragile}
Let $G(\mathcal{A})$ be a group generated by a reducible automaton. Then each $\mathcal{F}_k$ is a deterministic $|X|^k$-context-free language and
\[
 WP(G(\mathcal{A}):A;\mathcal{R})=\bigcup_{k>0} \mathcal{F}_k.
\]
\end{prop}
\begin{proof}
\indent Let us start by proving that each $\mathcal{F}_k$ is a deterministic $|X|^k$-context-free language.
\\
Let $A'=A\setminus\{id\}$ if $A$ contains the sink state $id$, otherwise we put $A'=A$. Let $D_{A'}$ be the Dyck language on the alphabet $A'$, i.e.:
$$
D_{A'}=\{w\in\wt{A'}^*: \theta(w)=\mathbf{1} \}, \quad \mbox{with } \theta:\wt{A'}^* \rightarrow F_A.
$$
It is a well known fact that $D_{A'}$ is a deterministic context-free language \cite{hop}. Let $S_A$ be the language obtained by the shuffle between $D_{A'}$ and $\{id,id^{-1}\}^*$, i.e., $S_A=\varepsilon^{-1}(D_{A'})$, where $\varepsilon:\wt{A}^\ast\to\wt{A'}^\ast$ is the morphism erasing all occurrences of $id,id^{-1}$. Roughly speaking, we are considering all possible insertions of letters $id,id^{-1}$ into elements in $D_{A'}$. Since deterministic context-free languages are closed under inverse homomorphisms \cite[Theorem~6.3]{hop}, we have that $S_A$ is also deterministic context-free. For every $u\in X^*$ define the language
\begin{equation}\label{defFu}
F_u=\{w\in \wt{A}^*: \theta(\varepsilon(w\cdot u))=\mathbf{1}\}.
\end{equation}
It follows from the definition that $F_u=\{w\in \wt{A}^*: w\cdot u\in S_A\}$. Let $\mathcal{B}=(\partial\mathcal{A})^{-}$ be the enriched dual automaton and consider its power $\mathcal{B}^{|u|}$, that is clearly a deterministic automaton. Let $T_u: \wt{A}^*\to \wt{A}^*$ be the mapping defined by $T_u(v)=u\circ v$ (where the action $\circ$ is referred to the one defined by $\mathcal{B}^{|u|}$). Observe that $F_u=T_u^{-1}(S_A)$, i.e., the set of words in $\wt{A}^{\ast}$ that label the input of paths starting at $u$ in $\mathcal{B}^{|u|}$ whose output belongs to $S_A$.  Since $S_A$ is deterministic context-free, then the language given by the words labeling its input is also deterministic context-free, as stated in \cite[Theorem 11.2]{hop}. Therefore we deduce that $F_u$ is a deterministic context-free language.  Now, we claim that $S_k:=\omega^{-1}(\St_{G(\mathcal{A})}(L_k))$ is a regular language. Indeed, for each $v\in L_k$ the language $\{w\in\wt{A}^*: w\circ v=v\}$ is regular since it is accepted by the automaton $\mathcal{B}^{k}$ with initial and final state $v$; therefore
\[
S_k=\bigcap_{v\in X^k} \{w\in\wt{A}^*: w\circ v=v\}
\]
is the intersection of regular languages, and it is also regular \cite[Theorem~3.3]{hop}. Finally
\begin{equation}\label{eq: intersection cf}
\mathcal{F}_k=S_k\cap\bigcap_{u\in X^{k}} F_u=\bigcap_{u\in X^{k}} (F_u\cap S_k)
\end{equation}
by virtue of Eq. \eqref{defSk} and Eq. \eqref{defFu}, so that $\mathcal{F}_k$ is the intersection of $|X|^k$ deterministic context-free languages, since the intersection of the deterministic context-free language $F_u$ with the regular language $S_k$ is still a deterministic context-free language \cite[Theorem~6.5]{hop}. Therefore, $\mathcal{F}_k$ is a deterministic $|X|^k$-context-free language.\\
\indent Now let $w\in  WP(G(\mathcal{A}):A;\mathcal{R})$: we clearly have $w\in \omega^{-1}(\St_{G(\mathcal{A})}(L_k))$ for all $k\ge 1$. If we put $m=|w|$, then by the reducibility of the automaton there is an $n_1$ such that $|\theta(w\cdot v)|<|\theta(w)|$ for all $v\in X^{\ge n_1}$, so eventually, after at most $m$ iterations of this argument, we may find $m$ positive integers $n_1,\ldots, n_m$ such that if $k=n_1+\cdots+ n_m$ it holds $\theta(w\cdot v)=\mathbf{1}$ for all $v\in X^{\ge k}$, i.e., $w\in \mathcal{F}_k$.\\
\indent Conversely we claim that, if $w\in  \mathcal{F}_k$ for some $k\geq 1$, then $w\in WP(G(\mathcal{A}):A;\mathcal{R})$. In order to prove that, it is enough to show that $w\circ u=u$ for all $u\in X^*$. Since $w\in \omega^{-1}(\St_{G(\mathcal{A})}(L_k))$ then $w\circ u'=u'$ for all $u'\in X^{\le k}$. Moreover, since $\theta(w\cdot u')=\mathbf{1}$ by hypothesis, we have that $w\cdot u'$ acts like the identity also on the subtree rooted at $u'$, thus $w\circ u'v=u'v$ for every $v\in X^*$ and so $w$ is in $WP(G(\mathcal{A}):A;\mathcal{R})$.\\
\end{proof}
Eq.~\eqref{eq: intersection cf} shows that $\mathcal{F}_k$ is the intersection of the deterministic context-free languages
\begin{equation}\label{nonbrutta}
F_u\cap S_k=\{w\in\omega^{-1}(\St_{G(\mathcal{A})}(L_k)): \theta(\varepsilon(w\cdot u))=\mathbf{1}\},
\end{equation}
with $u\in X^k$. We have the following lemma.

\begin{lemma}\label{lem: kernel of an homomorphism}
Let $u\in X^k$. There is a morphism $\varphi_u:H\to F_A$ from a finite index normal free subgroup $H$ of $F_A$ such that
\[
\ker(\varphi_u)=\theta(F_u\cap S_k).
\]
Moreover, $\theta^{-1}(\theta(F_u\cap S_k))=F_u\cap S_k$ holds.
\end{lemma}
\begin{proof}
First observe that $S_k=\omega^{-1}(\St_{G(\mathcal{A})}(L_k))$ is a submonoid of $\wt{A}^*$ and $H=\theta(S_k)$ is a finite index normal subgroup of $F_A$, since $\St_{G(\mathcal{A})}(L_k)$ is a finite index normal subgroup of $G(\mathcal{A})$. Let $\phi_u:S_k\to \wt{A}^*$ be the map defined as:
\begin{equation}\label{defphiu}
\phi_u(w)=\varepsilon(w\cdot u).
\end{equation}
We claim that this is a homomorphism (it is essentially the virtual endomorphism defined in \cite{nekrashevyc}). Indeed, for all $w_1, w_2\in S_k$, we have
\[
\phi_u(w_1\,w_2)=\varepsilon((w_1w_2)\cdot u)=\varepsilon((w_1\cdot u)\, (w_2\cdot u))=\varepsilon(w_1\cdot u)\varepsilon(w_2\cdot u)=\phi_u(w_1)\phi_u(w_2),
\]
since $w_i\circ u=u$. We now define $\varphi_u:H\to F_A$ by putting
$$
\varphi_u(g)=\theta(\phi_u(w)),
$$
where $w\in\wt{A}^*$ is any word with $\theta(w)=g$. Notice that, by the definition of $\phi_u$ in Eq. \eqref{defphiu}, the map $\varphi_u$ does not depend on the particular choice of $w$. Since
\begin{eqnarray}\label{hh-1}
(hh^{-1})\cdot u=(h\cdot u)(h\cdot u)^{-1}
\end{eqnarray}
holds for every $h\in S_k$, and, thus, $\phi_u(hh^{-1})=\phi_u(h)\phi_u(h)^{-1}=\mathbf{1}$ holds for every $h\in S_k$, we conclude that $\varphi_u$ is a well defined homomorphism. Moreover, it follows from the definition of $\varphi_u$ that
\[
\ker(\varphi_u)=\theta(\{w\in S_k: \theta(\varepsilon(w\cdot u))=\mathbf{1}\})=\theta(F_u\cap S_k),
\]
where the last equality follows from Eq. \eqref{nonbrutta}.\\
\indent
We conclude the proof by showing that $\theta^{-1}(\theta(F_u\cap S_k))=F_u\cap S_k$. Notice that it is enough to prove  that for any $u\in X^k$ one has $\theta^{-1}(\theta(F_u\cap S_k))\subseteq F_u\cap S_k$, since the other inclusion always holds.
Let $w\in \wt{A}^*$ be such that $w\in \theta^{-1}(\theta(F_u\cap S_k))$. Then $\theta(w)\in \theta(F_u\cap S_k)$, i.e., there exists $v\in F_u\cap S_k$ such that
$\theta(w)=\theta(v)$. This means that $w$ and $v$ only differ for factors of type $hh^{-1}$, then clearly $w$ and $v$ act in the same way on $X^{\ast}$, and in particular $w\in S_k$ because $v\in S_k$. On the other hand, Eq. \eqref{hh-1} ensures $\phi_u(w)=\phi_u(v)$. Hence $ \theta(\varepsilon(w\cdot u))= \theta(\varepsilon(v\cdot u))=\mathbf{1}$. In particular $w\in  F_u\cap S_k$.
\end{proof}

It follows from Lemma \ref{lem: kernel of an homomorphism} that
\begin{eqnarray}\label{aggiunta11feb}
\theta(\mathcal{F}_k)= \theta\left(\bigcap_{u\in X^k}(F_u\cap S_k)\right)=\bigcap_{u\in X^k}\theta (F_u\cap S_k) =\bigcap_{u\in X^k}\ker(\varphi_u),
\end{eqnarray}
where the second equality is a consequence of the property $\theta^{-1}(\theta(F_u\cap S_k))=F_u\cap S_k$. In particular, $\theta(\mathcal{F}_k)$ is a normal subgroup of $F_A$ since it is intersection of normal subgroups.

Therefore, we may define for each $k\ge 1$ the quotient group
\[
G_k=F_A/\theta(\mathcal{F}_k)
\]
that we call the \emph{$k$-th fragile group}. In view of Proposition~\ref{prop: decomposition by k-fragile} we have that each $k$-th fragile group $G_k$ is a poly-context-free group, in particular, it belongs to the class $\mathcal{BDG}$ defined in \cite{TCS}, consisting of the groups having a deterministic multipass word problem. Brough's conjecture states that groups whose word problem is a poly-context-free language are virtually subgroups of the direct product of free groups. To the best of the authors' knowledge, this conjecture is still open, and in Proposition \ref{prop: k-fragile group structure} we partially support it, by showing that it holds in the case of the fragile groups associated with a group defined by a reducible automaton. We first need the following fact.

\begin{lemma}\label{lem: intersection of languages}
Let $N=N_1\cap N_2\cap\ldots \cap N_k$ be a language that is intersection of $k$ languages on the alphabet $A$ with the property that $\theta^{-1} (\theta(N_i))=N_i$ for each $i=1,\ldots, k$. Suppose that for each $i=1,\ldots, k$ there is a morphism $\varphi_i:H_i\to F_{m_i}$ between a subgroup $H_i$ of $F_A$ of finite index and a free group $F_{m_i}$ of rank $m_i$ such that $\ker(\varphi_i)=\theta(N_i)$. Then the group $G=F_A/\theta(N)$ is virtually a subgroup of $F_{m_1}\times \cdots \times F_{m_k}$.
\end{lemma}
\begin{proof}
Since each $H_i$ is a finite index subgroup of $F_A$, then  $\mathcal{H}=H_1\cap \ldots \cap H_k$ is a finite index subgroup of $F_A$. Let $\phi:\mathcal{H}\to F_{m_1}\times \cdots \times F_{m_k}$ be the morphism defined component-wise by $\phi=(\varphi_1,\ldots, \varphi_k)$. This is a well defined homomorphism with
\[
 \ker(\phi)=\ker(\varphi_1)\cap \ldots \cap \ker(\varphi_k)=\theta(N_1)\cap \ldots \cap \theta(N_k)=\theta(N),
\]
where the last equality follows from the property $\theta^{-1} (\theta(N_i))=N_i$ for each $i=1,\ldots, k$, as in Eq. \eqref{aggiunta11feb}.
Hence $\mathcal{H}/\theta(N)$ is a subgroup of $F_{m_1}\times \cdots \times F_{m_k}$. Since $\mathcal{H}$ is a finite index subgroup of $F_A$ we have that $\mathcal{H}/\theta(N)$ is a finite index subgroup of $G=F_A/\theta(N)$. Therefore $G$ is virtually a subgroup of the direct product of free groups $F_{m_1}\times \cdots \times F_{m_k}$.
\end{proof}
We denote by $F_A^m$ the $m$-th iterated direct product of the free group $F_A$ with itself.
\begin{prop}\label{prop: k-fragile group structure}
Each $k$-th fragile group $G_k$ is virtually a subgroup of the direct product of $F_A^m$ with $m=|X|^k$.
\end{prop}
\begin{proof}
Let $u\in X^k$ and put $N_u=F_u\cap S_k$. Recall that $G_k=F_A/\theta(\mathcal{F}_k)$ and that $\theta(\mathcal{F}_k)=\bigcap_{u\in X^k}\theta(N_u)$ by Eq. \eqref{aggiunta11feb}. The claim follows from Lemma~\ref{lem: intersection of languages} applied to the morphisms $\varphi_u:H\to F_A$ with $\ker(\varphi_u)=\theta(F_u\cap S_k)$ as in Lemma~\ref{lem: kernel of an homomorphism}.
\end{proof}

Since $\mathcal{F}_i\subseteq \mathcal{F}_j$ holds for $j\ge i$, we have a family of naturally defined epimorphisms $\psi_{i,j}:G_i\to G_j$. Thus, the family of fragile groups $\{G_i\}_{i\in\mathbb{N}}$ together with the epimorphisms $\psi_{i,j}$ is a direct system of groups whose direct limit $\varinjlim G_i$ is $G(\mathcal{A})$ by virtue of Proposition~\ref{prop: decomposition by k-fragile}. Therefore Proposition~\ref{prop: k-fragile group structure} implies the following structural result.
\begin{theorem}\label{thmastrazeneca}
An automaton group $G(\mathcal{A})$ associated with a reducible invertible automaton $\mathcal{A}$ is isomorphic to the direct limit of groups which are virtually subgroups of the direct product of free groups.
\end{theorem}
Theorem \ref{thmastrazeneca} gives some insight on the general structure of $G(\mathcal{A})$ via its associated fragile groups. The following result probably belongs to the folklore and it has the same flavor of \cite[Proposition 1]{sidki2}. Recall that the Nielsen-Schreier theorem states that a subgroup of a free group is itself isomorphic to a free group. Moreover, a group $G$ is Hopfian if every epimorphism from $G$ to $G$
is an isomorphism (see, for instance, \cite{rob}).
\begin{prop}\label{prop: dichotomy}
Let $H$ be a subgroup of the direct product $F_{n_1}\times \cdots \times F_{n_k}$ of free groups. Then, either $H$ is abelian or it contains a free non-abelian subgroup.
\end{prop}
\begin{proof}
Let $\{g_i, i\in I\}$ be a set of generators for $H$, and let $p_j:H\to F_{n_j}$ be the projecting homomorphism onto the $j$-th component. We consider the following two mutually excluding cases.
\begin{itemize}
\item For each $j=1,\ldots, k$ there is an element $h_j\in F_{n_j}$ such that $p_j(g_i)\in\la h_j\ra $ for all $i\in I$. Note that this is equivalent to the fact that $\la p_j(g_i), i\in I \ra $ is a free subgroup of $F_{n_j}$ of rank $1$ by the Nielsen-Schreier theorem.
Thus, for every generator $g_i$, we have $g_i=(h_1^{s_1},\ldots, h_k^{s_k})$ for some integers $s_1, \ldots, s_k$. Therefore the generators $\{g_i, i\in I\}$ commute, as they commute component-wise, and so $H$ is abelian.
\item There is  some $j\in\{1,\ldots, k\}$ and $s,t\in I$ such that $\la p_j(g_s),p_j(g_t)\ra $ is a free subgroup of $F_{n_j}$ of rank $2$ by the Nielsen-Schreier theorem. Thus $\la g_s, g_t\ra$ is a free subgroup of rank 2 in $H$, since finitely generated free groups are Hopfian.
\end{itemize}
\end{proof}
The following theorem holds.
\begin{theorem}\label{thmnonome}
Let $\mathcal{A}$ be a reducible automaton and suppose that $G(\mathcal{A})$ has exponential growth. Then the fragile groups $\{G_{i}\}$ are all non-amenable.
\end{theorem}
\begin{proof}
Let $\psi_i:G_i\to G(\mathcal{A})$ be the canonical homomorphism. Notice that $\psi_i$ is surjective for any $i$ by Proposition \ref{prop: decomposition by k-fragile}. Since $G(\mathcal{A})$ has exponential growth, it cannot be the homomorphic image of a virtually abelian group. By using Proposition \ref{prop: k-fragile group structure} and Proposition~\ref{prop: dichotomy} we deduce that $G_i$ is a finite index subgroup of the direct product of free groups containing a free group of rank at least $2$. In particular, $G_i$ is not amenable.
\end{proof}

It follows from the previous theorem that, if $G(\mathcal{A})$ is the group associated with a reducible automaton and it has exponential growth, then it is a direct limit of non-amenable groups. This gives us important information about the presentation of $G(\mathcal{A})$.

\begin{theorem}\label{presentability}
Let $\mathcal{A}$ be a reducible automaton and suppose that $G(\mathcal{A})$ is amenable with exponential growth. Then $G(\mathcal{A})$ is not finitely presented.
\end{theorem}
\begin{proof}
Suppose, by contradiction, that $G(\mathcal{A})$ is finitely presented and then there exist $i$ relations $w_1,\ldots,w_i\in F_A$ such that $G(\mathcal{A})$ is isomorphic to $F_A/N$, where $N$ is the normal closure of
$w_1,\ldots,w_i$. By Proposition \ref{prop: decomposition by k-fragile} there exists $\bar{k}$ such that $w_1,\ldots,w_i\in \theta(\mathcal F_{\bar{k}})$: this implies that    $G(\mathcal{A})$ is isomorphic to the $\bar{k}$-th fragile group $G_{\bar{k}}$. This is a contradiction, because  by Theorem \ref{thmnonome} the group $G_{\bar{k}}$ is non-amenable.
\end{proof}

The previous theorem applies to a large class of automaton groups. In fact if $\mathcal{A}$ is a bounded automaton, then  $G(\mathcal{A})$ is amenable  \cite{amenability bounded}. Moreover, a graph automaton group  associated with a graph with at least two edges has exponential growth  \cite[Corollary 3.10]{articolo0} and,  by virtue of  Lemma~\ref{lem: loop reduces}, when the graph is a tree  the group is generated by a reducible automaton. The following corollary follows.

\begin{cor}\label{namemay}
Tree automaton groups associated with trees having at least two edges are not finitely presented and they are amenable groups obtained as a direct limit of non-amenable groups.
\end{cor}
We notice that this corollary shows that the presentation property of tree automaton groups is strongly different from the presentation property of tree automaton semigroup (compare with Theorem \ref{theo: submonoid is a trace monoid}). It also follows from Corollary \ref{namemay} that tree automaton groups constitute a class of groups that do not have free non-abelian subgroups although they are direct limit of groups containing free non-abelian groups.\\

\indent We conclude this section by exhibiting a procedure to construct a reducible automaton. Such a method defines a class of reducible automata that contains also tree automata and several more. Suppose that we want to build a reducible invertible automaton $\mathcal{A}$ on the set of states $A$ containing the sink state $id$. Given a finite index subgroup $H$ of $F_A$, the Schreier graph $Sch(H,A)$ of $H$ with respect to $A$ is defined as the graph whose vertices are the right cosets $Hg = \{ hg : h \in H \}$, for $g \in F_A$, and whose edges are of type $(Hg,Hga), a \in A$. In particular, one labels by $a$ the (oriented) edge $(Hg,Hga)$.
We associate with $H$ an invertible reduced automaton $\mathcal{B}$ by defining its enriched dual $(\partial\mathcal{B})^-$ as follows.
\begin{enumerate}
\item We first construct the Schreier graph $Sch(H,\wt{A})$ of $H$ with respect to the symmetric generating set $\wt{A}$.
\item Let $Y$ be a spanning tree of $Sch(H,\wt{A})$. The automaton $\partial\mathcal{B}^-$ is obtained from $Sch(H,\wt{A})$ by adding the outputs as follows:
\begin{itemize}
\item if the transition $p\mapright{a}q$ (and thus also its inverse $q\mapright{a^{-1}}p$) does not belong to $Y$, then we modify it by adding outputs as $p\vlongfr{a}{id}q$ and $q\vlongfr{a^{-1}}{id^{-1}}p$;
\item otherwise, we arbitrarily assign an output from $A$ in case $p\mapright{a}q$ belongs to the spanning tree $Y$.
\end{itemize}
\end{enumerate}
It is not difficult to see that any word $g\in F_A$ labeling a loop at some vertex $q$ of $(\partial\mathcal{B})^-$  has the property that $|q\circ g|<|g|$, thus $\mathcal{B}$ is reducible.

\end{document}